\documentclass{article}
\usepackage[margin=0.75in]{geometry}
\usepackage{hyperref}
\usepackage{amsmath, amscd, amsfonts, amsthm, amssymb, enumerate, graphicx}
\usepackage[shortlabels]{enumitem}
\usepackage{url}
\usepackage[all]{xy}

\newtheorem{thm}{Theorem}[section]
\newtheorem{lem}[thm]{Lemma}
\newtheorem{prop}[thm]{Proposition}

\theoremstyle{definition}

\newtheorem{defn}[thm]{Definition}
\newtheorem{examp}[thm]{Example}
\theoremstyle{remark}

\newtheorem{case}{Case}

\DeclareMathOperator{\ev}{ev}

\DeclareMathOperator{\cha}{char}

\DeclareMathOperator{\mult}{mult}

\DeclareMathOperator{\Frac}{Frac}

\DeclareMathOperator{\Cl}{Cl}
\DeclareMathOperator{\Gal}{Gal}
\DeclareMathOperator{\Disc}{Disc}

\DeclareMathOperator{\Pf}{Pfaff}
\newcommand{\red}{\mathrm{red}}
\newcommand{\<}{\left\langle}
\renewcommand{\>}{\right\rangle}
\renewcommand{\(}{\left(}
\renewcommand{\)}{\right)}
\newcommand{\GL}{\mathrm{GL}}
\newcommand{\SL}{\mathrm{SL}}

\newcommand{\FF}{\mathbb{F}}

\newcommand{\ZZ}{\mathbb{Z}}
\renewcommand{\aa}{\mathfrak{a}}

\newcommand{\cc}{\mathfrak{c}}

\newcommand{\pp}{\mathfrak{p}}

\newcommand{\OO}{\mathcal{O}}

\newcommand{\cross}{\times}
\newcommand{\tensor}{\otimes}

\newcommand{\textand}{\quad \text{and} \quad}

\renewcommand{\to}{\mathop{\rightarrow}\limits}
\newcommand{\bx}{\square}
\newcommand{\etop}{e_{\mathrm{top}}}
\newcommand{\ftop}{f_{\mathrm{top}}}
\newcommand{\intsec}{\cap}
\newcommand{\union}{\cup}
\newcommand{\isom}{\cong}
\newcommand{\ignore}[1]{}

\newcommand{\bbq}[8]{
\begin{minipage}{0.1\linewidth}
\xymatrix@!0{
& #5 \ar@{-}[rr]\ar@{-}[dd]
& & #6 \ar@{-}[dd]
\\
#1 \ar@{-}[ur]\ar@{-}[rr]\ar@{-}[dd]
& & #2 \ar@{-}[ur]\ar@{-}[dd]
\\
& #7 \ar@{-}[rr]
& & #8
\\
#3 \ar@{-}[rr]\ar@{-}[ur]
& & #4 \ar@{-}[ur]
}
\end{minipage}
}

\newdir^{ (}{{}*!/-5pt/@^{(}}
\newdir_{ (}{{}*!/-5pt/@_{(}}

\begin{document}

\title{Quintic rings over Dedekind domains\\and their sextic resolvents}
\author{Evan M. O'Dorney%
\let\thefootnote\relax\footnote{\emph{University of Notre Dame, Notre Dame, Indiana, USA.} \url{eodorney@nd.edu}}}
\maketitle

\begin{abstract}
Bhargava parametrized quintic rings over $\ZZ$ by quadruples of $5\times 5$ alternating matrices. We extend the construction to work similarly over any Dedekind domain $R$. No assumptions are needed on the characteristic of $R$. The resolvent consists of a pair of locally free modules $L$, $M$ with two multilinear maps between them; we can view $L$ as $Q/R$, for $Q$ the quintic ring, and $M$ as $S/R$, where $S$ is a sextic resolvent ring. As in Bhargava's treatment, any quintic ring has a resolvent ring, and for a maximal ring, the resolvent is unique. We hope that this work will enable the removal of the condition that the characteristic be different from $2$ in Bhargava-Shankar-Wang's proof of Linnik's conjecture on the asymptotic distribution of discriminants of relative extensions.
\end{abstract}

\begin{itemize}
  \item Keywords: quintic ring, Dedekind domain, resolvent, higher composition laws
  \item MSC: 13F05, 13B02 (Primary), 11E76, 11R21 (Secondary)
  \item Data availability statement: This project has no associated data.
\end{itemize}

\section{Introduction}

Since their publication in the 2000's \cite{B1,B2,B3,B4}, there has been continual interest in Bhargava's \emph{higher composition laws,} that is, parametrizations of number rings and related objects by $n$-ary $d$-ic forms and similarly explicit objects. Early examples of this, like the Levi-Delone-Gross 
parametrization of cubic rings by their index forms \cite{Levi, DF, GGS}, now fit into a larger paradigm that include the 14 higher composition laws appearing in Bhargava's initial series as well as numerous additional examples, such as Wood's parametrization of $2$-torsion in rings parametrized by an odd-degree binary form by $2\times n \times n$ matrices \cite{W2xnxn}. These higher composition laws have been found very useful, especially in applying the geometry-of-numbers method in arithmetic statistics \cite{Bd4c,Bd5c}. 

Although Bhargava's paper series worked out only the case where the base ring is $\ZZ$, it has immediately been clear that his methods work over more general bases. Wood worked out two parametrizations over a general \emph{base scheme:} the case of Gauss composition of binary quadratic forms, which gave higher composition laws their name \cite{WGauss}, and the parametrization of quartic rings, or shall we say fourfold covers \cite{WQuartic}.

A more modest degree of generalization, in which still a great degree of concreteness can be achieved, is that when the base $R$ is a Dedekind domain. This includes the case $R = \OO_K$ of the ring of integers of a number field, allowing one to study relative extensions, which have long been of interest in arithmetic statistics \cite{EVnumber,Dummit_Counting}. The Dedekind case also includes $R$ that are fields, DVR's, or coordinate rings of affine curves, which arise in assorted contexts. Special note must be given to the work of Bhargava, Shankar, and Wang \cite{BSW}, which gives the asymptotic number of relative extensions of degree up to $5$ over a global field not of characteristic $2$, proving what is commonly called Linnik's conjecture in those cases. The present work suggests that the last hypothesis can be removed.

In \cite{ORings}, the author adapted Bhargava's work to parametrize quadratic, cubic, and quartic rings over a Dedekind domain. Among the remaining cases, the parametrization of quintic rings stands out as the most involved of Bhargava's higher composition laws, and one of great interest:

\begin{thm}[Bhargava \cite{B4}, Theorems 1 and 2 and Corollary 3] There is a suitable notion of resolvent for quintic rings over $\ZZ$ with the following properties:
\begin{enumerate}[$($a$)$]
  \item There is a canonical bijection between the orbits of $\Gamma = \GL_4(\ZZ) \cross \SL_5(\ZZ)$ on the space $\ZZ^4 \tensor \Lambda^2 \ZZ^5$ of quadruples of $5 \times 5$ alternating matrices and the set of isomorphism classes of pairs $(Q,S)$, where $Q$ is a quintic ring and $S$ is a sextic resolvent ring of $Q$.
  \item Every quintic ring $Q$ has a resolvent, that is, appears in this bijection for some element $A \in \ZZ^4 \tensor \Lambda^2 \ZZ^5$.
  \item If $Q$ is a maximal ring, then the resolvent is unique up to isomorphism, and the element $A$ is unique up to $\Gamma$-equivalence.
\end{enumerate}
\end{thm}

In this paper, we prove a generalization of these results to the case where the base ring $\ZZ$ is replaced by a Dedekind domain. Certain modifications in the method of Bhargava \cite{B4} must be made to address the ideal class group of $R$ and the fact that $R$ may have characteristic two. Recall that finitely generated, locally free modules $M$ over a Dedekind domain are characterized by two invariants: the \emph{dimension} $n = \dim_K (M \tensor K)$ (where $K = \Frac R$) and the \emph{Steinitz class,} that is, the ideal class $\aa$ for which $\Lambda^n M \isom \aa$ as $R$-modules.
The main theorem is notably unchanged in spirit:

\begin{thm}\label{thm:main}
There is a suitable notion of resolvent for quintic rings over a Dedekind domain $R$ with the following properties:
\begin{enumerate}[$($a$)$]
  \item\label{main:param} Let $\aa \in \Cl(R)$. Let $L_\aa = R^{\oplus 3} \oplus \aa$ be the lattice over $R$ of dimension $4$ and Steinitz class $\aa$. Let $M_\aa = R^{\oplus 4} \oplus \aa^{3}$ be the lattice over $R$ of dimension $5$ and Steinitz class $\aa^{3}$. There is a canonical bijection between:
  \begin{itemize}
    \item the orbits of $\Gamma_\aa = \GL(L_\aa) \cross \GL(M_\aa)$ on the space $\aa \tensor L_\aa \tensor (\Lambda^2 M_\aa)^*$, whose elements can be viewed as quadruples of $5 \times 5$ alternating matrices whose entries lie in certain powers of $\aa$; and
    \item the isomorphism classes of pairs $(Q,S)$, where $Q$ is a quintic ring and $S$ is a sextic resolvent ring of $Q$.
  \end{itemize}
  \item\label{main:rsv} Every quintic ring $Q$ has a resolvent, that is, appears in this bijection for some element $A \in \ZZ^4 \tensor \Lambda^2 \ZZ^5$.
  \item\label{main:max} If $Q$ is a maximal ring, then the resolvent is unique up to isomorphism, and the element $A$ is thus unique up to $\Gamma$-equivalence.
\end{enumerate}
\end{thm}

The remainder of the paper is structured as follows. In Section \ref{sec:def_rsv}, we define a suitable notion of resolvent. In Sections \ref{sec:rsv_to_ring} and \ref{sec:constr_rsv}, we prove Theorem \ref{thm:main} by making the passage from resolvents to rings and from rings to resolvents, respectively. In Section \ref{sec:sextic_ring} we touch on the ring structure on the sextic resolvent, which is noticeably absent from our definition. We close with a bound on the number of resolvents of a ring (Section \ref{sec:bounds}) and with some examples (Section \ref{sec:examples}).

\section{Defining resolvents}\label{sec:def_rsv}
Let $Q$ be a quintic ring over a Dedekind domain $R$, and let $L = Q/R$. Our first task is to generalize the notion of a \emph{sextic resolvent,} developed by Bhargava in \cite{B4} in the case $R = \ZZ$. Following the approach of Wood \cite{WQuartic} and the author \cite{ORings}, we expect the resolvent to consist of a rank-$5$ lattice $M$ (originating as $S/R$, where $S$ is a sextic ring) with two elements of some modules derived by multilinear constructions from $L$ and $M$. The orientation map $\theta$, which relates the top exterior powers of $L$ and $M$, is easy to guess. The discriminant of an $R$-algebra $T$ naturally lies in $(\Lambda^\mathrm{top}(T))^{\tensor-2}$. Just as the equality $\Disc Q = \Disc C$ between the discriminants of a quartic ring and its cubic resolvent(s) suggests an identification of the top exterior powers of the two rings, so the relation $\Disc S = (16 \Disc Q)^3$ (Bhargava's (33) of \cite{B4}) linking the discriminants of a quintic ring and its sextic resolvent(s) suggests an isomorphism
\[
  \theta : \Lambda^5 M \to (\Lambda^4 L)^{\tensor 3}.
\]
The second piece of data, that which contains the $40$ integers that actually parametrize resolvents over $\ZZ$, is slightly trickier to adapt. Bhargava presents it as a map $\phi$ from $L$ to $\Lambda^2 M$ (equivalently, from $\Lambda^2 M^*$ to $L^*$), but this does not have the correct properties in our situation. The correct construction, foreshadowed somewhat by the mysterious constant factor in Bhargava's fundamental resolvent ((28) in \cite{B4}), is to take a map
\[
  \phi : \Lambda^4 L \tensor L \to \Lambda^2 M.
\]
Finally, we must find the fundamental relations that link $\phi$ and $\theta$ to the ring structure. Just as Lemma 9 of \cite{B3} provided the inspiration for Bhargava's coordinate-free description of resolvents of a quartic ring (\cite{B3}, section 3.9), so we begin at Lemma 4(a), which, after eliminating the references to $S_5$-closure, states that
\[
  \frac{1}{2} \cdot \omega \cdot \left(
    \Pf\begin{bmatrix}
      \phi(y) & \phi(x) \\
      \phi(x) & \phi(z)
    \end{bmatrix} - 
    \Pf\begin{bmatrix}
      \phi(y) & \phi(x) \\
      \phi(x) & -\phi(z)
    \end{bmatrix}
  \right) = 1 \wedge y \wedge x \wedge z \wedge yz
\]
for a certain generator $\omega$. The Pfaffians are to be interpreted by writing $\phi(x)$, etc., as a $5\times 5$ alternating matrix with regard to any convenient basis (i.e.~viewing it as an alternating bilinear form on $\Lambda^2 M$, once a generator of $\Lambda^4 L$ is fixed). Then we paste together four of these to make a $10\times 10$ alternating matrix and take the Pfaffian. This is a clever way to manufacture certain degree-$5$ integer polynomials in the $40$ coefficients of $\phi$. To re-express them in a way that is coordinate-free and applicable in characteristic $2$, we consider two preliminary multilinear constructions.

\subsection{The quadratic map $\bx$}
Let $V$ be a $5$-dimensional vector space over a field $K$ (which we will soon take to be $\Frac R$). We examine the constructions that can be made starting with elements of $\Lambda^2 V$. We have a bilinear map $\wedge : \Lambda^2 V \cross \Lambda^2 V \to \Lambda^4 V$. However, the most fundamental map from $\Lambda^2 V$ to $\Lambda^4 V$ is not the bilinear map $\wedge$ but the quadratic map from which it arises. It is defined by
\begin{equation}\label{eq:def box}
  \left(\sum_{i=1}^n v_i \wedge w_i \right)^\bx = \sum_{1 \leq i < j \leq n} v_i \wedge w_i \wedge v_j \wedge w_j.
\end{equation}
It is not hard to prove that this is well defined. Note that if $\cha K \neq 2$, then $\bx$ can be described more simply by
\[
  \mu^\bx = \frac{1}{2} \mu \wedge \mu;
\]
if $\cha K = 2$, then $\mu \wedge \mu = 0$ yet $\bx$ is nonzero.
Moreover, the bilinear map $\wedge$ can always be recovered from $\bx$ via
\begin{equation}
  \mu \wedge \nu = (\mu + \nu)^\bx - \mu^\bx - \nu^\bx.
\end{equation}

\subsection{The contraction $\ev$}
The second construction takes one element $\mu \in \Lambda^2 V$ and two elements $\alpha,\beta \in \Lambda^4 V$ and outputs an element of a suitable one-dimensional vector space as follows. First, the perfect pairing \[
  \wedge : \Lambda^4 V \cross V \to \Lambda^5 V
\]
allows us to identify $\alpha$ and $\beta$ as elements of $\Lambda^5 V \tensor V^*$. These have a wedge product
\[
  \alpha \wedge \beta \in \Lambda^2(\Lambda^5 V \tensor V^*) \cong (\Lambda^5 V)^{\tensor 2} \tensor \Lambda^2 V^*.
\]
We now use the duality between $\Lambda^2 V^*$ and $\Lambda^2 V$, described explicitly by
\[
  (f \wedge g)(v \wedge w) = fv \cdot gw - fw \cdot gv,
\]
to obtain an element
\[
  \ev(\mu;\alpha,\beta) \in (\Lambda^5 V)^{\tensor 2}.
\]
The contraction $\ev$ is linear in each of its three arguments, and alternating in the last two.
\subsection{The definition}

We are now ready to state the definition of a sextic resolvent.
\begin{defn}
Let $Q$ be a quintic ring over a Dedekind domain $R$, and let $L = Q/R$. A \emph{resolvent} for $Q$ consists of a rank-$5$ lattice $M$ and a pair of linear maps
\[
  \phi : \Lambda^4 L \tensor L \to \Lambda^2 M \textand \theta : \Lambda^5 M \to (\Lambda^4 L)^{\tensor 3},
\]
with $\theta$ an isomorphism, satisfying the identity
\begin{equation}\label{eq:res}
  \theta^{\tensor 2}[\ev(\phi(\lambda_1x);\phi(\lambda_2y)^\bx,\phi(\lambda_3z)^\bx)]
    = \lambda_1 \lambda_2^2 \lambda_3^2 (x \wedge y \wedge z \wedge yz)
\end{equation}
where $x,y,z \in L$ and $\lambda_i \in \Lambda^4 L$ are formal variables.
\end{defn}

Note that the expression within square brackets lies in $(\Lambda^5 M)^{\tensor 2}$; applying $\theta^{\tensor 2}$, one ends up in $(\Lambda^4 L)^{\tensor 6}$ which is where the right-hand side also resides. It should also be remarked that the product $yz$ is the unique appearance of the ring structure of $Q$; translating the lifts $\tilde{y}, \tilde{z}$ by constants in $R$ simply changes the product $\tilde{y}\tilde{z}$ by multiples of $\tilde{y}$, $\tilde{z}$, and $1$, thereby not changing the product $y \wedge z \wedge yz$.

\section{Resolvent to ring}\label{sec:rsv_to_ring}
Our first task is to show that the resolvent maps $\phi$ and $\theta$ uniquely encode the multiplication data of the ring $Q$.

\begin{lem} \label{lem:quintic ring}
Let $L$ and $M$ be lattices over $R$ of ranks $4$ and $5$ respectively, and let $\phi : \Lambda^4 L \tensor L \to \Lambda^2 M$ and $\theta : \Lambda^5 M \to (\Lambda^4 L)^{\tensor 3}$ be maps. There is a quintic ring $Q$ with a quotient map $Q/R \cong L$, unique up to isomorphism, such that $(M,\phi,\theta)$ is a resolvent of $Q$.
\end{lem}
\begin{proof}
Let $(e_1,e_2,e_3,e_4)$ be a basis for $L$, by which we mean that there is a decomposition $L = \aa_1e_1 \oplus \cdots \oplus \aa_4e_4$ for some fractional ideals $\aa_i$ of $R$. (Because the Steinitz class is $\aa = \prod \aa_i$, we could take $\aa_1=\aa_2=\aa_3=(1)$ and $\aa_4 = \aa$; but we refrain from this choice for the sake of symmetry.) To place a ring structure on the module $Q = L \oplus R$, it is then necessary to choose the coefficients $c_{ij}^k \in \aa_k\aa_i^{-1}\aa_j^{-1}$ such that
\[
  e_i e_j = \sum_k c_{ij}^k e_k,
\]
with the conventions $e_0 = 1$ and $\aa_0 = (1)$. Note that the $c_{ij}^k$ with $i = 0$ or $j = 0$ are already known. Hence the ring structure is given by the $50$ coefficients $c_{ij}^k$, $1 \leq i \leq j \leq 4$, $0 \leq k \leq 4$.

Some of these coefficients are immediately determined by the resolvent. For instance, if $\{i,j,k,\ell\}$ is a permutation of $\{1,2,3,4\}$, and $\epsilon = \pm 1$ its sign, then we know
\begin{equation}\label{eq:cijk}
  c_{ij}^k = -\epsilon\etop^{-1} \cdot e_\ell \wedge e_i \wedge e_j \wedge e_ie_j = -\epsilon\omega \cdot \theta^{\tensor 2}[\ev(\phi(e_\ell\etop);\phi(e_i\etop)^\bx,\phi(e_j\etop)^\bx)],
\end{equation}
where $\etop = e_1 \wedge e_2 \wedge e_3 \wedge e_4 = \epsilon \cdot e_i \wedge e_j \wedge e_k \wedge e_\ell$ is the generator of $\Lambda^4 L$ induced by the chosen basis and
\[
  \omega = \theta^*({\etop^*}^3) \in \big((\Lambda^5 M)^*\big)^{\tensor2}.
\]
This determines the values of all $c_{ij}^k$ where $i$, $j$, and $k$ are nonzero and distinct.

Likewise, the following expressions are determined, for $i$, $j$, $k$, $\ell$ distinct:
\begin{equation}\label{eq:cijk2}
\begin{aligned}
  c_{ii}^j &= 
  \epsilon\etop^{-1} \cdot e_\ell \wedge e_i \wedge (e_i + e_k) \wedge e_i(e_i + e_k) - c_{ik}^j \\
  c_{ik}^k - c_{ij}^j &= \epsilon\etop^{-1} \cdot e_\ell \wedge e_i \wedge (e_j + e_k) \wedge e_i(e_j + e_k) - c_{ik}^j + c_{ij}^k \\
  c_{ii}^i - c_{ij}^j - c_{ik}^k &= \epsilon\etop^{-1} \cdot e_\ell \wedge (e_i + e_k) \wedge (e_i + e_j) \wedge (e_i + e_j)(e_i + e_k) \\
  &\quad - c_{jk}^i + c_{ik}^j + c_{ij}^k + c_{ii}^j + c_{ii}^k + (c_{kj}^j - c_{ki}^i) + (c_{jk}^k - c_{ji}^i).
\end{aligned}
\end{equation}
The reader familiar with ring parametrizations will recognize the left-hand sides of \eqref{eq:cijk} and \eqref{eq:cijk2} 
as the linear expressions in the $c_{ij}^k$ that are invariant under translations $e_i \mapsto e_i + t_i$ ($t_i \in \aa_i^{-1}$) of the ring basis elements (see \cite[(21)]{B4}). If we normalize our basis so that, say, $c_{12}^1 = c_{12}^2 = c_{34}^3 = c_{34}^4 = 0$, then all the $c_{ij}^k$ are now uniquely determined, except for the $c_{ij}^0$. The $c_{ij}^0$ can be computed by comparing the coefficients of $k$ in $(e_ie_j)e_k$ and $e_i(e_je_k)$ for any $k \neq i$, yielding formula (22) of \cite{B4}:
\[
  c_{ij}^0 = \sum_{r=1}^4 (c_{jk}^r c_{ri}^k - c_{ij}^r c_{rk}^k).
\]
The lemma is now reduced to three verifications.
\begin{enumerate}
  \item That all $c_{ij}^k$ belong to the correct ideals $\aa_k\aa_i^{-1}\aa_j^{-1}$. This is routine.
  \item That the $c_{ij}^0$ are well defined, and more generally that the associative law holds on the ring $Q = \sum \aa_i e_i$ that we have just constructed. This is a collection of integer polynomial identities in the $40$ free coefficients of $\phi$ in the chosen basis; as such, it was proved in the course of Bhargava's parametrization of quintic rings over $\ZZ$.
  \item That the original maps $\phi$ and $\theta$ indeed form a resolvent of $Q$, i.e.~that the identity \eqref{eq:res} holds. Since \eqref{eq:res}
   does not directly generalize any result of Bhargava, we here give the outline of a proof. We can assume that $\lambda_1 = \lambda_2 = \lambda_3 = \etop$ and $x$ is a basis element $e_\ell$, since the equation \eqref{eq:res} is linear in those variables. We can also assume that each of $y$ and $z$ is a basis element or a sum of two different basis elements, since \eqref{eq:res} is quadratic in those variables. Now we have a finite set of cases, some of which are the relations \eqref{eq:cijk} and \eqref{eq:cijk2}. The others will be reduced to them using the following properties of the underlying multilinear operations:
\begin{lem}\label{lem:boxplay}
Let $V$ be a $5$-dimensional vector space, and let $\mu,\nu,\xi \in \Lambda^2 V$ and $\alpha \in \Lambda^4 V$. Then
\begin{enumerate}[label=$(\alph*)$,ref=(\alph*)]
  \item \label{mna} $\ev(\mu; \mu \wedge \nu, \alpha) = -{\ev}(\nu; \mu^\bx, \alpha)$
  \item \label{ma} $\ev(\mu; \mu^\bx, \alpha) = 0$
  \item \label{mnx} $\ev(\nu;\mu^\bx, \mu \wedge \xi) = -{\ev}(\xi; \mu^\bx, \mu \wedge \nu)$.
\end{enumerate}
\end{lem}
\begin{proof}
For \ref{mna}, set $\mu = u\wedge v + w\wedge x$ and $\nu = y \wedge z$ (both sides being linear in $\nu$) and expand. The difference of the two sides is found to be alternating in $u,v,w,x,y,z$, hence zero, since $\Lambda^6 V = 0$. Then \ref{ma} follows by setting $\mu = \nu$, and \ref{mnx} by the derivation
\[
  \ev(\nu; \mu^\bx, \mu \wedge \xi) = -{\ev}(\mu; \mu \wedge \nu, \mu \wedge \xi)
  = \ev(\mu; \mu \wedge \xi, \mu \wedge \nu) = -{\ev}(\xi; \mu^\bx, \mu \wedge \nu).
  \qedhere
\]
\end{proof}
Now we return to proving
\begin{equation}\label{eq:res2}
  \theta^{\tensor 2}\left[{\ev}\(\phi(\etop x); \phi(\etop y)^\bx,\phi(\etop z)^\bx\)\right]
    = \etop^5 (x \wedge y \wedge z \wedge yz)
\end{equation}
for $x = e_\ell$ and $y,z \in \{e_i\}_i \union \{e_i + e_j\}_{i<j}$. The cases where $e_\ell$ does not appear in $y$ or $z$ are all subsumed by the definitions \eqref{eq:cijk} and \eqref{eq:cijk2}, with one exception: the expression for $c_{ii}^j$ is not visibly symmetric under switching $k$ and $\ell$. This can be seen by writing
\begin{align*}
  c_{ii}^j &= \epsilon\etop^{-1} (e_\ell \wedge e_i \wedge (e_i + e_k) \wedge e_i(e_i + e_k)
    - e_\ell \wedge e_i \wedge e_k \wedge e_ie_k) \\
  &= \epsilon\etop^{-5} \big({\ev}(\phi(e_\ell); \phi(e_i)^\bx,\phi(e_i + e_k)^\bx) - \ev(\phi(e_\ell); \phi(e_i)^\bx,\phi(e_k)^\bx)\big) \\
  &= \epsilon\etop^{-5} \big({\ev}(\phi(e_\ell);\phi(e_i)^\bx,\phi(e_i)^\bx + \phi(e_i) \wedge \phi(e_k) + \phi(e_k)^\bx) - \ev(\phi(e_\ell); \phi(e_i)^\bx,\phi(e_k)^\bx)\big) \\
  &= \epsilon\etop^{-5} \big({\ev}(\phi(e_\ell; \phi(e_i)^\bx,\phi(e_i) \wedge \phi(e_k))\big)
\end{align*}
and using Lemma \ref{lem:boxplay}\ref{mnx}.
It remains to dispose of the cases where $e_\ell$ does appear in $y$ or $z$. We prove them by induction on the total number of $e$ terms in $x$, $y$, and $z$, the base cases being those already shown. Suppose that $x = e_\ell$ appears in $x + y = e_\ell + e_k$ (the case where $x$ appears in $z$ is symmetric). For brevity let $\lambda = \phi(\etop x)$, $\mu = \phi(\etop y)$, $\nu = \phi(\etop z)$. Then
\begin{align*}
   &\theta^{\tensor 2}\left[\ev\(\phi(\etop x); \phi(\etop (x + y))^\bx,\phi(\etop z)^\bx\)\right] \\ 
   &=\theta^{\tensor 2}\left[\ev\(\lambda; (\lambda + \mu)^\bx, \nu^\bx \)\right] \\
   &=\theta^{\tensor 2}\left[\ev\(\lambda; \lambda^\bx + \lambda \wedge \mu + \mu^\bx, \nu^\bx \)\right] \\
   &= \theta^{\tensor 2}\left[-{\ev\(\mu; \lambda^\bx, \nu^\bx\)} + \ev\(\lambda; \mu^\bx, \nu^\bx\)\right] \\
   &= \etop^5 (-y \wedge x \wedge z \wedge xz + x \wedge y \wedge z \wedge yz) \\
   &= \etop^5 (x \wedge (x + y) \wedge z \wedge (x + y)z),
\end{align*}
where the induction hypothesis applies since $(x,y,z)$ and $(y,x,z)$ both have fewer total $e$ terms than $(x, x + y, z)$. \qedhere
\end{enumerate}
\end{proof}

\subsection{Omission of $\theta$}
With our definition, it is natural to wonder what happens when the datum $\theta$ is changed. The answer is simple:
\begin{prop}
If $\theta$ is scaled by a unit $\gamma \in R^\cross$, the corresponding quintic ring is unchanged, up to isomorphism.
\end{prop}
\begin{proof}
More strongly, the \emph{resolvent} itself is unchanged up to isomorphism. Observe that the scalar multiplications $\gamma^2 : L \to L$ and $\gamma^5 : M \to M$ take one resolvent to the other, thanks to the commutative diagrams:
\begin{equation}
\xymatrix{
  \Lambda^4 L \ar[d]_{\Lambda^{4}(\gamma^2) \tensor \gamma^2 = \gamma^{10}} \tensor L \ar[r]^\phi & \Lambda^2 M \ar[d]_{\Lambda^2 (\gamma^5) = \gamma^{10}} & \Lambda^5 M \ar[d]^{\Lambda^{5} \gamma^5 = \gamma^{25}} \ar[r]^{\gamma \theta} & (\Lambda^4 L)^{\tensor 3} \ar[d]^{\(\Lambda^{4}(\gamma^2)\)^{\tensor 3} = \gamma^{24}} \\
  \Lambda^4 L \tensor L \ar[r]^\phi & \Lambda^2 M & \Lambda^5 M \ar[r]^{\theta} & (\Lambda^4 L)^{\tensor 3}
}
\qedhere
\end{equation}
\end{proof}
In spite of this, we retain the datum $\theta$ in our definition of resolvent, as it makes all the resolvent conditions polynomial relations, without an existential quantifier, and eases base change.
\begin{proof}[Proof of Theorem \ref{thm:main}\ref{main:param}:]
Let $(L, M, \phi, \theta)$ be a resolvent of Steinitz class $\aa$. The modules $L \isom L_\aa$ and $M \isom M_\aa$ are known up to isomorphism. If isomorphisms are chosen, then $\phi \in (\Lambda^4 L \tensor L)^2 \tensor \Lambda^2 M \isom \aa^{-1} \tensor L_\aa^* \tensor \Lambda^2 M_{\aa}$ is realized as an element of the desired lattice, unique up to $\GL(L) \cross \GL(M)$. Conversely, if such a $\phi$ is given, then due to the redundancy of $\theta$, we get a resolvent unique up to isomorphism.
\end{proof}

\subsection{Compatibility with Bhargava's definitions}

If $\aa = (1)$, that is, $L$ and $M$ are free over $R$, the resolvent devolves into the basis representation of $\phi$. This has $40$ independent entries which can be arranged into a quadruple of $5\times 5$ alternating matrices, representing the values $\phi(x)$ (as $x$ runs through a basis of $L$) as alternating bilinear forms on $M^*$. The coefficients $c_{ij}^k$ of the ring we have constructed are certain degree-$5$ polynomials in these $40$ entries which are easily identified with the formulas given in (21) of \cite{B4}. Thus our definition of resolvent is compatible with Bhargava's (Definition 10), which justifies our invocation of his computations in our situation, despite the dissimilarities of the definitions.

\section{Constructing resolvents}\label{sec:constr_rsv}
We now wish to go the other way and prove Theorem \ref{thm:main}\ref{main:rsv}: every quintic ring over a Dedekind domain admits at least one resolvent. We begin with the case where $R = K$ is a field.

In the quartic case \cite{B3,ORings}, it was the \emph{trivial ring} $T = K[x,y,z]/(x,y,z)^2$ that had the largest family, all other rings having a unique resolvent. Likewise, here we separate out some of the most degenerate rings, those of the last three types $A_{18},A_{19},A_{20}$ in the classification of Mazzola \cite[p.292]{Maz80}:
\begin{defn}
A quintic algebra $Q$ over $K$ is \emph{very degenerate} if it has subspaces $Q_4 \subseteq Q_3$, of dimension $4$ and $3$ respectively, such that $Q_4Q_3 = 0$ (that is, the product of any element of $Q_4$ and any element of $Q_3$ is zero). A quintic algebra $Q$ over $R$ is \emph{very degenerate} if the corresponding $K$-algebra $Q \tensor_R K$ is.
\end{defn}
Let $Q$ be a very degenerate quintic ring over $R$. Upon taking a suitable basis, the multiplication table of $Q$ has the form
\begin{equation} \label{eq:very deg table}
\begin{tabular}{c|ccccc}
$\times$ & $1$ & $\xi_1$ & $\xi_2$ & $\xi_3$ & $\xi_4$ \\ \hline
$1$ & $1$ & $\xi_1$ & $\xi_2$ & $\xi_3$ & $\xi_4$ \\
$\xi_1$ & $\xi_1$ & $c_{11}^1\xi_1 + c_{11}^2\xi_2$ & $0$ & $0$ & $0$ \\
$\xi_2$ & $\xi_2$ & $0$ & $0$ & $0$ & $0$ \\
$\xi_3$ & $\xi_3$ & $0$ & $0$ & $0$ & $0$ \\
$\xi_4$ & $\xi_4$ & $0$ & $0$ & $0$ & $0$
\end{tabular}
\end{equation}
in which at most two of the structure constants $c_{ij}^k$ are nonzero. It is easy to compute from the definition that $Q$ has a family of resolvents of the form
\[
A = \left(
\begin{bmatrix}
 0 & * & *         & *        & -1 \\
 * & 0 & 0         & 0        & 0  \\
 * & 0 & 0         & c_{11}^2 & 0  \\
 * & 0 & -c_{11}^2 & 0        & 0  \\
 1 & 0 & 0         & 0        & 0  
\end{bmatrix},
\begin{bmatrix}
  0 & * & *         & *        & 0 \\
  * & 0 & 0         & 0        & 0 \\
  * & 0 & 0         & c_{11}^1 & 0 \\
  * & 0 & -c_{11}^1 & 0        & 0 \\
  0 & 0 & 0         & 0        & 0
\end{bmatrix},
\begin{bmatrix}
  0 & * & * & * & 0 \\
  * & 0 & 0 & -1 & 0 \\
  * & 0 & 0 & 0 & 0 \\
  * & 1 & 0 & 0 & 0 \\
  0 & 0 & 0 & 0 & 0 \\
\end{bmatrix},
\begin{bmatrix}
  0 & * & * & * & 0 \\
  * & 0 & 1 & 0 & 0 \\
  * & -1 & 0 & 0 & 0 \\
  * & 0 & 0 & 0 & 0 \\
  0 & 0 & 0 & 0 & 0 \\
\end{bmatrix}
\right),
\]
with respect to bases for $L_\aa \isom R \oplus R \oplus R \oplus \aa, M_\aa \isom \aa \oplus \aa \oplus \aa \oplus R \oplus R$, where $*$ denotes any element of the appropriate ideal. We can therefore ignore very degenerate rings in the sequel.
\subsection{Resolvents over a field}
\begin{thm}\label{thm:field}
Every not very degenerate quintic $K$-algebra has a unique resolvent up to isomorphism.
\end{thm}
\begin{proof}
Let $M$ be a $K$-vector space of dimension $5$, and let $\theta : \Lambda^5 M \to (\Lambda^4 L)^{\tensor 3}$ be any isomorphism. So far we have not made any choices. We will first construct the map $\phi^\bx = \phi(\bullet)^\bx$, a quadratic map from $\Lambda^4 L \tensor L$ to $\Lambda^4 M$. For this purpose we concoct a corollary of \eqref{eq:res} that involves only $\phi^\bx$.
\begin{lem}\label{lem:tobox}
Let $V$ be a $5$-dimensional vector space. Let $\mu \in \Lambda^2 V$ and $\alpha,\beta,\gamma,\delta \in \Lambda^4 V$. Then
\[
  \mu^\bx \wedge \alpha \wedge \beta \wedge \gamma \wedge \delta =
  \ev(\mu; \alpha, \beta) \ev(\mu; \gamma, \delta)
  + \ev(\mu; \alpha, \gamma) \ev(\mu; \delta, \beta)
  + \ev(\mu; \alpha, \delta) \ev(\mu; \beta, \gamma)
\]
in $\Lambda^5(\Lambda^4 V) \cong (\Lambda^5 V)^{\tensor 4}$.
\end{lem}
\begin{proof}
Write the general $\mu$ as $u\wedge v + w\wedge x$ ($u,v,w,x\in V$) and expand.
\end{proof}

Motivated by this, we define for any quintic ring $Q$ the pentaquadratic form
\begin{equation}\label{eq:def F}
\begin{aligned}
  F(a,b,c,d,e)
  &= (a \wedge b \wedge c \wedge bc)(a \wedge d \wedge e \wedge de) \\
  &\quad{} + (a \wedge b \wedge d \wedge bd)(a \wedge e \wedge c \wedge ec)
  + (a \wedge b \wedge e \wedge be)(a \wedge c \wedge d \wedge cd)
\end{aligned}
\end{equation}
from $L^5$ to $(\Lambda^4 L)^{\tensor 2}$, or equivalently from $(\Lambda^4 L \tensor L)^5$ to $(\Lambda^4 L)^{\tensor 12}$. We get that for any resolvent $(M,\phi,\theta)$ of $Q$,
\begin{equation} \label{eq:prop F}
  \theta^{\tensor 4}(\phi(a)^\bx \wedge \phi(b)^\bx \wedge \phi(c)^\bx \wedge \phi(d)^\bx \wedge \phi(e)^\bx) = F(a,b,c,d,e).
\end{equation}
We claim the following:
\begin{lem}\label{lem:very deg}
$F$ is identically zero if and only if $Q$ is very degenerate.
\end{lem}
\begin{proof}
We prove that the property of being very degenerate is invariant under base-changing to the algebraic closure $\bar{K}$ of $K$; then the lemma can be proved by checking the finitely many quintic algebras over an algebraically closed field (see Mazzola \cite{Maz80} and Poonen \cite{PoonenIsomorphism}). Let $\bar{Q} = Q \tensor_K \bar{K}$ be the corresponding $\bar{K}$-algebra. Clearly if $Q$ is very degenerate, so is $\bar{Q}$, so assume that $\bar{Q}$ is very degenerate. 

First look at the reduced $Q_\red$, $\bar Q_\red$ formed by quotienting out by the nilpotents. Since $\bar Q$ is very degenerate, $\bar Q_\red$ is isomorphic to either $\bar K$ or $\bar K \cross \bar K$, the latter case occurring when
\begin{equation}\label{eq:1+4}
  \bar Q \isom \bar K \cross \bar K[\epsilon_1, \epsilon_2, \epsilon_3]/\<\epsilon_1, \epsilon_2, \epsilon_3\>^2.
\end{equation}
\begin{case}
If $\bar Q_\red \isom \bar K$, then $Q_\red$ must be a field, and its degree must divide $5$: so $Q_\red = Q$ or $Q_\red \isom K$. If $Q_\red = Q$, then $\bar Q$ is either totally split or (in the purely inseparable case, which occurs only in characteristic $5$) $\bar Q \isom \bar K[\epsilon]/\<\epsilon^5\>$, which is not a very degenerate ring. If $Q_\red \isom K$, then there is a distinguished section $s : L \hookrightarrow Q$, namely the isomorphism onto the $4$-dimensional subspace of nilpotents. The condition that $Q$ be very degenerate can be stated as saying that the multiplication tensor
\begin{align*}
  \operatorname{mult} &\in L^* \tensor L^* \tensor L \\
  \mult(x,y) &= s(x)s(y) \bmod K
\end{align*}
has rank $1$, that is, is an elementary tensor. This is a condition invariant under base change (incidentally, it is given by an intersection of quadrics which are the coefficients of $F$).
\end{case}
\begin{case}
  If $\bar Q_\red \isom \bar K \cross \bar K$, then we have \eqref{eq:1+4}. As the rank-$1$ idempotent is unique, $Q \cong K \cross Q'$ must also split as a product of factors of the correct degrees. Then $Q$ is very degenerate if and only if $Q'$ is the trivial ring, that is, all the entries of its multiplication table are $0$, and this too is invariant under base change. \qedhere
\end{case}
\end{proof}
Picking $a_1,\ldots,a_5 \in \Lambda^4 L \tensor L$ such that $F(a_1,a_2,a_3,a_4,a_5) = f_0 \neq 0$, we get that the five vectors $v_i = \phi(a_i)^\bx$ must form a basis such that
\[
  \theta^{\tensor 4}(v_1 \wedge v_2 \wedge v_3 \wedge v_4 \wedge v_5) = f_0.
\]
Any such basis is as good as any other: they are all related by elements of $\SL(\Lambda^4 M)$, which is canonically isomorphic to $\SL(M)$. Once the $v_i$ are fixed, there is at most one candidate for the map $\phi^\bx$ up to $\SL(M)$-equivalence, namely
\begin{equation}\label{eq:phibox}
  \phi(a)^\bx = \frac{1}{f_0} \sum_{i=1}^5 F(a_1,\ldots, \hat{a}_i, \ldots, a_5) v_i
\end{equation}

Then the relations
\[
  \ev(\phi(x); \phi(a_i)^\bx,\phi(a_j)^\bx) = x \wedge a_i \wedge a_j \wedge a_ia_j,
\]
for $1 \leq i < j \leq 5$, determine the map $\phi$ uniquely. So the resolvent map $\phi$, if it exists, is unique. It remains to verify the resolvent relations, which are finite in number since the $a_i$ in \eqref{eq:phibox} can be chosen from the finite set $\{e_1,e_2,e_3,e_4,e_1+e_2,e_1+e_3,\ldots,e_3+e_4\}$ for any basis $\{e_1,e_2,e_3,e_4\}$ of $\Lambda^4 L \tensor L$. It remains to prove that the $(M,\phi,\theta)$ we have hereby constructed is actually a resolvent; this is a collection of integer polynomial identities, not in a family of free variables as in the previous lemma, but in the coefficients $c_{ij}^k$ of the given ring $Q$, which are restricted by the associative law. To prove these identities, it is enough to change base to the algebraic closure $\bar{K}$ and exhibit a resolvent for each of the finitely many quintic $\bar K$-algebras found in the classification of Mazzola and Poonen. For $\bar{K}^{\oplus 5}$, the unique nondegenerate quintic $\bar{K}$-algebra, the resolvent is shown in Example \ref{ex:first}. This takes care of all the $Q$ which are limits of \'etale algebras in the variety of based algebras (which is true for all $Q$ in characteristic $0$, as Mazzola observes, and is most likely true in characteristic $p$).
\end{proof}

\subsection{Resolvents over a Dedekind domain}
\label{sec:ded res}
We now want to endow our resolvents with integral structure.

\begin{proof}
Let $Q$ be a quintic ring over a Dedekind domain $R$. We will assume that $Q$ is not very degenerate and hence that the corresponding $K$-algebra $Q_K = Q\tensor_R K$ has a unique resolvent $(M_K, \phi,\theta)$. Resolvents of $Q$ are now in bijection with lattices $M$ in the vector space $M_K$ such that
\begin{gather}
  \phi(\Lambda^4 L \tensor L) \subseteq \Lambda^2 M \label{eq:phicond} \\
  \theta(\Lambda^5 M) \subseteq (\Lambda^4 L)^{\tensor 3}. \label{eq:thetacond}
\end{gather}

For any resolvent $M$, note that we must have
\[
  M^* \cong \Lambda^4 M \tensor (M^5)^{\tensor -1} \supseteq \<\phi^\bx(\Lambda^4 L \tensor L)\> \tensor (\theta((\Lambda^4 L)^{\tensor 3}))^{\tensor -1}.
\]
Here $\<\phi^\bx(\Lambda^4 L \tensor L)\>$ means the closure of the image of the quadratic map $\phi^\bx$ under addition and $\OO_K$-multiplication. Since $Q$ is not very degenerate, the right-hand side is a lattice of full rank and we may take its dual, which we denote by $M_0$. Then any resolvent is contained in $M_0$. Condition \eqref{eq:phicond} is vacuous for $M = M_0$, since
  \[
    \phi(\lambda x)(\phi(\lambda' y)^\bx,\phi(\lambda'' z)^\bx)
    = \theta^{\tensor 2}(\lambda\lambda'\lambda''(x \wedge y \wedge z \wedge yz)) \in (\theta(\Lambda^3 L))^{\tensor 2}
  \]
  for all $\lambda', \lambda'' \in \Lambda^3 L$ and $y,z \in L$. On the other hand, condition \eqref{eq:thetacond} is generally not satisfied by $M = M_0$; indeed, one readily finds that $\theta^{-1}((\Lambda^4 L)^{\tensor 3}) \subseteq \Lambda^5 M_0$ using \eqref{eq:prop F}.

The classification of resolvents is now reduced to a local problem. Any $M$ determines a family of resolvents $(M_\pp,\phi,\theta)$ of the quintic algebras $Q_\pp$ over the DVR's $R_\pp \subseteq K$, and conversely an arbitrary choice of resolvents $M_\pp$ of the $R_\pp$ can be glued together to form the resolvent $M = \bigcap_\pp M_\pp$. The choice $M_\pp = M_{0,\pp} = M_0 \tensor R_\pp$ is forced for all but finitely many primes $\pp$, namely those dividing the ideal
\begin{equation} \label{eq:cc}
  \cc = [\Lambda^5 M_0 : \theta^{-1}((\Lambda^4 L)^{\tensor 3})] = [(\Lambda^4 L)^{\tensor 2} : \<F(a,b,c,d,e) : a,b,c,d,e \in L\>].
\end{equation}
Therefore, for the remainder of the proof, we assume that $R = R_\pp$ is a DVR. We adapt the method of proof of Bhargava \cite{B4}, Lemmas 13--15.

We first prove that for some $n \geq 0$, the module $\pi^n L$, which corresponds to the ring $R + \pi^{n} Q$, has a resolvent. Let $M_1$ be any lattice of the correct index $\cc$ in $M_0$. With respect to any bases of $L$ and $M_0$, the map $\phi$ is represented by a box $A = (A_1, A_2, A_3, A_4)$ whose entries are in $K$. If we pass from $L$ to $\pi^{5} L$ and from $M_1$ to $\pi^{12} M_1$, the discriminant condition remains satisfied and the entries are multiplied by $\pi$. Performing this operation enough times, the entries become integral. 

We now attempt to lower the exponent $n$ for which $\pi^n L$ has a resolvent $M$ until it becomes zero.

First, the quadratic map $\phi^\bx : L \to \Lambda^4 V \isom \Lambda^5 V \tensor V^*$ is determined by its values at the ten vectors $e_1, \ldots, e_4, e_1 + e_2, \ldots, e_3 + e_4$, where the $e_i$ form a basis of $L$. Wedging any five of them yields a vector in $\Lambda^5(\Lambda^4 V) \cong (\Lambda^5 V)^{\tensor 4}$; these are the $\binom{10}{5} = 252$ \emph{determinantal invariants} of \cite{B4}. Being invariant under $\SL(V)$, they are quadratic polynomials in the fundamental invariants for the action of $\SL(V)$ on boxes $A$, namely the ring coefficients $c_{ij}^k$. Of interest to us is that these polynomials have coefficients in $\ZZ$. Since $L$ defines a ring, the $c_{ij}^k$ of $\pi^n$ are divisible by $\pi^n$ and the determinantal invariants are divisible by $\pi^{2n}$, that is, lie in $\Lambda^5(M)^{\tensor 4}$.

Let $A = (A_1, A_2, A_3, A_4)$ be the $4 \times 5 \times 5$ box, with entries in $R$, corresponding to $\pi^{2n} L$ and its resolvent $M$. Since the determinantal invariants are all linearly dependent mod $\pi$, the image of $\tilde \phi^\bx = \pi^{2n} \phi^\bx$ lies in a sublattice of $M_1$ of index $\pi$; that is, with respect to a suitable basis, the upper $4\times 4$ submatrices of the $A_i$ and their $R$-linear combinations are all singular modulo $\pi$. As explained in the proof of Lemma 15 in \cite{B4}, it is possible to change bases further so that these submatrices are either all $0$ mod $\pi$ or have one of the forms
\begin{gather}
  \left(
  \begin{bmatrix}
    0 & 1 & 0 & 0 \\
    -1 & 0 & 0 & 0 \\
    0 & 0 & 0 & 0 \\
    0 & 0 & 0 & 0
  \end{bmatrix},
  \begin{bmatrix}
    0 & 0 & c_1 & 0 \\
    0 & 0 & 0 & 0 \\
    -c_1 & 0 & 0 & 0 \\
    0 & 0 & 0 & 0
  \end{bmatrix},
  \begin{bmatrix}
    0 & 0 & 0 & c_2 \\
    0 & 0 & 0 & 0 \\
    0 & 0 & 0 & 0 \\
    -c_2 & 0 & 0 & 0
  \end{bmatrix},
\begin{bmatrix}
0 & 0 & 0 & 0 \\
0 & 0 & 0 & 0 \\
0 & 0 & 0 & 0 \\
0 & 0 & 0 & 0
\end{bmatrix}
\right) \label{eq:red1} \\
\left(
\begin{bmatrix}
0 & 1 & 0 & 0 \\
-1 & 0 & 0 & 0 \\
0 & 0 & 0 & 0 \\
0 & 0 & 0 & 0
\end{bmatrix},
\begin{bmatrix}
0 & 0 & 1 & 0 \\
0 & 0 & 0 & 0 \\
-1 & 0 & 0 & 0 \\
0 & 0 & 0 & 0
\end{bmatrix},
\begin{bmatrix}
0 & 0 & 0 & 0 \\
0 & 0 & 1 & 0 \\
0 & -1 & 0 & 0 \\
0 & 0 & 0 & 0
\end{bmatrix},
\begin{bmatrix}
0 & 0 & 0 & 0 \\
0 & 0 & 0 & 0 \\
0 & 0 & 0 & 0 \\
0 & 0 & 0 & 0
\end{bmatrix}
\right). \label{eq:red2}
\end{gather}
If the submatrices are zero or of the form \eqref{eq:red1} modulo $\pi$, we use the fact that the lower right $3\times 3$ submatrices of each matrix in $A$ is divisible by $p$ to derive that $\pi^{n-1}L$ has the resolvent
\[
M' = \pi^{-3}\<\mu_1, \pi\mu_2, \pi\mu_3, \pi\mu_4, \mu_5\>
\]
where $\{\mu_i\}$ is the basis of $M$ in which $A$ has been written. In the case $\eqref{eq:red2}$, from $p \mid c_{23}^4 = \ev(\phi(e_1); \phi(e_2)^\bx, \pi(e_3)^\bx)$ we get that the $(4,5)$ entry of $A_1$ is zero mod $p$. Symmetrically we get the same for $A_2$ and $A_3$, and also for $A_4$ since replacing $A_2$ by $A_2 + A_4$ does not change the situation. So
the entire fourth rows and fourth columns of $A$ are divisible by $p$, so $\pi^{n-2}L$ has the resolvent
\[
M' = \pi^{-5}\<\mu_1, \mu_2, \mu_3, \pi\mu_4, \mu_5\>.
\]

Thus, as long as $n > 0$, we can lower $n$. In the case that $n$ becomes negative, we can raise it back to $0$ by the method of proof of Lemma 13 of \cite{B4}.
\end{proof}

In the special case that $\cc$ is the unit ideal, $M_0$ is the only resolvent. This occurs in one important instance, highlighted in Theorem \ref{thm:main}\ref{main:max}.
\begin{lem} \label{lem:maximal}
If $Q$ is a \emph{maximal} quintic ring, that is, is not contained in any strictly larger quintic ring, then the ideal $\cc$ in \eqref{eq:cc} is the unit ideal, implying that $Q$ has a unique resolvent.
\end{lem}
\begin{proof}
Suppose that $\cc$ were not the unit ideal, so there is a prime $\pp$ such that $\pp | F(a,b,c,d,e)$ for all $a,b,c,d,e \in L$. We will prove that $Q$ is not maximal at $\pp$. It is convenient to localize and assume that $R = R_\pp$ is a DVR with uniformizer $\pi$.

Note that $Q/\pp Q$, a quintic algebra over $R/\pp$, has its associated pentaquadratic form $F$ identically zero, so by Lemma \ref{lem:very deg}, it is very degenerate. So $Q$ has an $R$-basis $(1,x,\epsilon_1,\epsilon_2,\epsilon_3)$ such that $(x,\epsilon_1,\epsilon_2,\epsilon_3)(\epsilon_1,\epsilon_2,\epsilon_3) \subseteq \pp R$. We claim that the lattice $Q'$ with basis $(1,x,\pi^{-1}\epsilon_1,\pi^{-1}\epsilon_2,\pi^{-1}\epsilon_3)$ either is a quintic ring or is contained in a quintic ring, showing that $Q$ is not maximal.

Set $M = \<\pi,x,\epsilon_1,\epsilon_2,\epsilon_3\>$ and $N = \<\pi,\pi x, \epsilon_1,\epsilon_2,\epsilon_3\>$. Then $Q \supseteq M \supseteq N \supseteq \pi Q$ and $MN \subseteq \pi Q$. Consider, for any $i,j \in \{1,2,3\}$, the multiplication maps
\[
\xymatrix{
  {Q}/{N} \ar[r]^{\epsilon_i} \ar@{=}[d] &
  {N}/{\pi Q} \ar[r]^{\epsilon_j} \ar@{=}[d] & 
  {\pi Q}/{\pi N} \ar[r]^{\epsilon_i} \ar@{=}[d] &
  {\pi N}/{\pi^2 Q} \ar@{=}[d] \\
  \<1,x\> & \<\epsilon_1,\epsilon_2,\epsilon_3\> & \<\pi,\pi x\> & \<\pi\epsilon_1,\pi\epsilon_2,\pi\epsilon_3\>.
}
\]
These are all linear maps of $R/\pp$-vector spaces. Denote by $f$ the composition of the left two maps and by $g$ the composition of the right two. Write $f(1) = \pi(a + bx)$, where $a,b \in R/\pp$. Then $g(\epsilon_i) = a\pi\epsilon_i$, since $x\epsilon_i \in \pi Q$. Thus $g$ is given in the bases above by the scalar matrix $a$. But $g$ has rank at most $2$, since it factors through the two-dimensional space $\pi Q/\pi N$; hence $a = 0$. So $N^2 \subseteq \pi M$.

Now consider the following multiplication maps:
\[
\xymatrix{
  {Q}/{M} \ar[r]^{\epsilon_i} \ar@{=}[d] &
  {N}/{\pi Q} \ar[r]^{\epsilon_j} \ar@{=}[d] &
  {\pi M}/{\pi N} \ar[r]^{\epsilon_k} \ar@{=}[d] &
  {\pi^2 Q}/{\pi^2 M} \ar[r]^{\epsilon_i} \ar@{=}[d] &
  {\pi^2 N}/{\pi^3 Q} \ar@{=}[d] \\
  \<1\> & \<\epsilon_1,\epsilon_2,\epsilon_3\> & \<\pi x\> & \<\pi^2\> & \<\pi^2\epsilon_1,\pi^2\epsilon_2,\pi^2\epsilon_3\>.
}
\]
Similarly to the previous argument, the composition of the first three maps must be zero, or else the composition of the last three would be a nonzero scalar. Since the images of the first map (as $i$ varies) span $N/\pi Q$, the composition of the middle two maps is always zero. Since $j$ and $k$ can vary independently and $\pi M/\pi N$ is one-dimensional, there are two cases:
\begin{enumerate}[(a)]
  \item The second map is always zero, that is, $N^2 \subseteq \pi N$. This implies that $\pi^{-1} N$ is a quintic ring, as desired.
  \item The third map is always zero, that is, $MN \subseteq \pi M$. We get that $\pi^{-1}\epsilon_i$ is integral over $R$ (look at the characteristic polynomial of its action on $M$), so $R[\pi^{-1}\epsilon_1,\pi^{-1}\epsilon_2,\pi^{-1}\epsilon_3]$ is finitely generated and thus a quintic ring, as desired.
\end{enumerate}
\end{proof}
Note that, in this proof, if the resolvent is not unique, then the extension $Q' \supsetneq Q$ has $(R/\pp)^3 \subseteq Q'/Q$. So the following stronger theorem holds:
\begin{thm}\label{thm:strong maximal}
If $Q$ is a quintic ring such that the $R/\pp$-vector space of congruence classes in $\pi^{-1}Q/Q$ whose elements are integral over $R$ has dimension at most $2$, for each prime $\pp$, then $Q$ has a unique resolvent.
\end{thm}

\section{The sextic ring}\label{sec:sextic_ring}
Given any resolvent $(L,M,\phi,\theta)$, the rank-$6$ lattice $S = M \oplus R$ also picks up a canonical ring structure, whose structure coefficients $d_{ij}^k$ are integer polynomials in the coefficients of $\phi$ of degree $12$ (for $k\neq 0$) and $24$ (for $k = 0$). As the construction given by Bhargava in \cite{B4}, Section 6 works without change over a Dedekind domain, we will not spell out the details. We have the equation
\begin{equation} \label{eq:disc}
  \Disc S = (16 \Disc Q)^3
\end{equation}
from (33) of \cite{B4}. (These discriminants are to be interpreted as specifying both the Steinitz class and the discriminant ideal; see \cite{ORings} for details.)

It is natural to ask whether the sextic resolvent ring is always an order in the sextic resolvent $K$-algebra, generated by classical methods from the theory of solving equations. For instance, if $Q$ is an order in $K^5$, is $S$ an order in $K^6$? We leave out the case where $\cha K = 2$, where \eqref{eq:disc} shows that $S$ is always degenerate.

\begin{thm}
Assume that $\cha K \neq 2$. Consider the familiar bijection between $n$-ic rings that are \'etale (that is, have nonzero discriminant) and maps $\Gal(\bar K/K) \to S_n$ to the symmetric group (see for instance Milne \cite[Theorem 7.29]{MilneFields}). Let $Q$ be a quintic ring and $S$ its sextic resolvent. Then the maps $\psi_Q$, $\psi_S$ associated to the \'etale $K$-algebras $Q \tensor_R K$ and $S \tensor_R K$ are related by the commutative diagram
\begin{equation}
  \xymatrix{
    \Gal(\bar K/K) \ar[d]_{\psi_Q} \ar[dr]^{\psi_S} \\
    S_5 \ar[r]^{\iota_{5,6}} & S_6
  }
\end{equation}
where $\iota_{5,6} : S_5 \to S_6$ is the exceptional embedding (given by composing the obvious injection with the famous outer automorphism of $S_6$).
\end{thm}
\begin{proof}
We may assume $R = K$ is a field. The proof consists of the following steps:
\begin{itemize}
  \item Check that the resolvent of $\ZZ^5$ is an order of index $64$ in $\ZZ^6$ (Example \ref{ex:first}).
  \item Deduce that the resolvent of $\bar K^5$ is $\bar K^6$.
  \item Analyze how the $S_n$-actions on $\bar K^n$ interact with the resolvent. We find that for each $\sigma \in S_5$,
  \[
    \sigma : \bar K^5 \to \bar K^5, \quad \iota_{5,6}\sigma : \bar K^6 \to \bar K^6
  \]
  act by permutation.
  \item By the standard description of the Galois parametrization (see Milne \cite[p.~107]{MilneFields}),
  \[
    Q = \{x \in \bar K^5 : gx = \psi_Q(g) x \quad \forall g \in \Gal(\bar K / K)\},
  \]
  where $g$ acts componentwise on $\bar K^5$, but $\psi_Q(g)$ acts by permutation.
  Consider the sextic algebra associated to $\iota_{5,6} \circ \psi_Q$:
  \[
    S' = \{x \in \bar K^6 : gx = \iota_{5,6}\big(\psi_Q(g)\big) x \quad \forall g \in \Gal(\bar K / K)\}.
  \]
  By the preceding considerations, $\phi$ and $\theta$ restrict to maps making $S'$ a resolvent for $Q$. Because $Q$ is nondegenerate over a field, the resolvent $S' = S$ is unique, so $\psi_S = \iota_{5,6} \circ \psi_Q$ as desired. \qedhere
\end{itemize}
\end{proof}

\section{Finiteness of the number of resolvents}\label{sec:bounds}

It is natural to wonder, if the resolvent of a quintic ring is not unique, how close to being unique it is. For the quartic case, a beautifully simple formula was given in \cite{B3} and extended to the Dedekind case in \cite{ORings}: the number of (numerical) resolvents is the sum of the absolute norms of the \emph{content}. In the quintic case, things do not appear to be so simple. However, we have finiteness and an upper bound that is theoretically effective. 
\begin{thm}
Let $Q$ be a not very degenerate quintic ring. Then $Q$ has at most
\[
  \prod_{\substack{\pp\text{ prime,}\\ \pp \mid \cc}} \left( \frac{N(\pp)^5 - 1}{N(\pp) - 1} \right)^{v_\pp(\cc)}
\]
resolvents, provided that the absolute norms $N(\pp) = |R/\pp|$ are finite. In particular, a not very degenerate quintic ring over the ring of integers of a number field has finitely many resolvents.
\end{thm}
\begin{proof}
Since all resolvents have index $\cc$ in $M_0$, it suffices to bound the number of sublattices of index $\cc$ in a fixed lattice $M_0$. By localization we may reduce to the case $\cc = \pp^n$, where $\pp$ is prime. Now a fixed lattice $M$ has $(N(\pp)^5 - 1)/(N(\pp) - 1)$ sublattices of index $\pp$, the kernels of the nonzero linear functionals $\ell : M/\pp M \to R/\pp$ mod scaling. A sublattice $M_n$ of index $\pp^n$ has a filtration $M_0 \subsetneq M_1 \subsetneq \cdots \subsetneq M_n$ where the quotients are $R/\pp$; given $M_i$, there are at most $(N(\pp)^5 - 1)/(N(\pp) - 1)$ possibilities for $M_{i+1}$, giving the claimed bound.
\end{proof}

\section{Examples}\label{sec:examples}
\begin{examp} \label{ex:first}
The most fundamental example of a sextic resolvent is as follows. Let $Q = R^{\oplus 5}$, with basis $e_1,e_2,\ldots, e_5$, and let $M = R^5$ with basis $f_1,\ldots, f_5$. Then the map
\[
  \phi(e_i) = f_i \wedge (f_{i-1} + f_{i+1})
\]
(indices mod $5$), supplemented by the orientation $\theta(\ftop) = \etop^3$ induced by these bases, is verified to be a resolvent for $Q$ (indeed the unique one, as $Q$ is maximal). The automorphism group $S_5$ of $Q$ acts on $M$ by the $5$-dimensional irreducible representation obtained (in characteristic not $2$) by restricting to the image of the exceptional embedding $\iota_{5,6}$ the standard representation of $S_6$, permuting the six vectors
\[
  f_{i-2} - f_{i-1} + f_i - f_{i+1} + f_{i+2} \quad (1 \leq i \leq 5) \textand f_1 + f_2 + f_3 + f_4 + f_5.
\]
The corresponding ring structure $S$ produced in Section \ref{sec:sextic_ring} is none other than the ring $S = S_\ZZ \tensor_\ZZ R$, where
\[
  S_\ZZ = \{(x_1;x_2;x_3;x_4;x_5;x_6) \in \ZZ^6 : x_i \equiv x_j \mod 2 \quad \forall i,j; \text{ and } \sum x_i \equiv 2x_1 \mod 4\}.
\]
\end{examp}
\begin{examp}
Consider the subring
\[
  Q = \{x_1e_1 + \cdots + x_5e_5 \in \ZZ^{\oplus 5} : x_1 \equiv x_2 \equiv x_3 \equiv x_4 \bmod p\}.
\]
Over $K$, this is a special case of the preceding example, so the $K$-resolvent $M_K$ and the associated maps $\phi$ and $\theta$ are forced. But the bounding module $M_0$ of Section \ref{sec:ded res} is no longer a resolvent, as can be seen by observing that $Q/pQ \cong \FF_p[t,\epsilon_1,\epsilon_2,\epsilon_3]/\langle\{t^2 - t, t\epsilon_i, \epsilon_i\epsilon_j\}\rangle$ is very degenerate. We have $L = \langle pe_1, pe_2, pe_3, e_5 \rangle$ and thus $\Lambda^4 L = \langle p^3 \etop \rangle$. One computes that
\[
  M_0 = \<p(f_1 + f_4),p^2 f_2, p^2 f_3, p^2 f_4, pf_5\>,
\]
and thus
\[
  \cc = [\Lambda^5 M_0 : \theta^{-1}((\Lambda^4 L)^{\tensor 3})]
  = [\langle p^8 \ftop \rangle : \langle p^9 \ftop \rangle] = p.
\]
Consequently a resolvent of $Q$ is a submodule $M$ of index $p$ in $M_0$ having the property that $\phi(\Lambda^4L \tensor L) \subseteq \Lambda^2 M$. Writing $M$ as the kernel of some linear functional $\ell : M_0/pM_0 \to \FF_p$, the condition is that $\ell$ lies in the kernel of each of the alternating bilinear forms obtained by reducing $\phi(x) \in \Lambda^2 M_0$ mod $p$ for all $x \in \Lambda^4L \tensor L)$. Let
\[
  f_1' = p(f_1 + f_4), f_2' = p^2 f_2, f_3' = p^2 f_3, f_4' = p^2 f_4, f_5' = pf_5
\]
be the basis elements of $M_0$ listed above. We compute
\begin{align*}
  \phi(p^4 \etop e_1) &= (pf_1' - f_4') \wedge (pf_5' + f_2') \\
  \phi(p^4 \etop e_2) &= f_2' \wedge (pf_1' - f_4' + f_3') \\
  \phi(p^4 \etop e_3) &= f_3' \wedge (pf_2' + f_4') \\
  \phi(p^3 \etop e_5) &= f_5' \wedge (pf_1').
\end{align*}
So, letting $\bar{f}'_i$ denote the basis vector of $M_0/pM_0$ corresponding to $f_i'$ and $\bar{f}'^*_i$ the corresponding vector of the dual basis, we have
\[
  \ell \in \ker(\bar{f}'_2 \wedge \bar{f}'_4) \intsec \ker(\bar{f}'_2 \wedge \bar{f}'_3)
  \intsec \ker(\bar{f}'_3 \wedge \bar{f}'_4) = \<\bar{f}'^*_1, \bar{f}'^*_5\>.
\]
Since $\ell$ can take any value in the last-named vector space, up to scaling, we get $p+1$ resolvents.
\end{examp}
\begin{examp}
The ring
\[
  Q = \ZZ \oplus \ZZ \oplus \ZZ[x,y]/(x,y)^2
\]
is a curious example of Theorem \ref{thm:strong maximal}. Although $Q$ is infinitely far from being maximal ($\ZZ \oplus \ZZ \oplus \ZZ[n^{-1}x,n^{-1}y]/(n^{-2}(x,y)^2)$ is a quintic extension ring for any $n > 0$), the extensions are only in two directions, as it were, and the resolvent is accordingly unique.
\end{examp}
\bibliography{../Master.bib}
\bibliographystyle{plain}

\end{document}